\documentclass[11pt]{amsart}

\usepackage{enumerate,url,amssymb,  mathrsfs}

\newtheorem{theorem}{Theorem}[section]
\newtheorem{lemma}[theorem]{Lemma}
\newtheorem*{lemma*}{Lemma}
\newtheorem{proposition}[theorem]{Proposition}
\newtheorem{corollary}[theorem]{Corollary}

\theoremstyle{definition}
\newtheorem{definition}[theorem]{Definition}
\newtheorem{example}[theorem]{Example}
\newtheorem{conjecture}[theorem]{Conjecture}

\theoremstyle{remark}
\newtheorem{remark}[theorem]{Remark}

\numberwithin{equation}{section}

\newcommand{\abs}[1]{\lvert#1\rvert}
\newcommand{\norm}[1]{\lVert#1\rVert}

\newcommand{\A}{\mathbb{A}}
\newcommand{\C}{\mathbb{C}}
\newcommand{\D}{\partial}
\newcommand{\K}{\mathcal{K}}
\newcommand{\LL}{\mathcal{L}}
\newcommand{\R}{\mathbb{R}}
\newcommand{\Z}{\mathbb{Z}}
\newcommand{\Ho}{\mathscr{H}}
\newcommand{\Ha}{\mathcal{H}}
\newcommand{\As}{\mathcal{A}}
\newcommand{\const}{\mathrm{const}}
\newcommand{\dtext}{\textnormal d}

\DeclareMathOperator{\re}{Re}
\DeclareMathOperator{\im}{Im}

\DeclareMathOperator{\Mod}{Mod}
\def\Xint#1{\mathchoice
{\XXint\displaystyle\textstyle{#1}}%
{\XXint\textstyle\scriptstyle{#1}}%
{\XXint\scriptstyle\scriptscriptstyle{#1}}%
{\XXint\scriptscriptstyle\scriptscriptstyle{#1}}%
\!\int}
\def\XXint#1#2#3{{\setbox0=\hbox{$#1{#2#3}{\int}$}
\vcenter{\hbox{$#2#3$}}\kern-.5\wd0}}

\def\dashint{\Xint-}
\def\le{\leqslant}
\def\ge{\geqslant}

\begin{document}

\title[Harmonic mappings of an annulus]{Harmonic mappings of an annulus,\\ Nitsche conjecture and its generalizations}

\author{Tadeusz Iwaniec}
\address{Department of Mathematics, Syracuse University, Syracuse,
NY 13244, USA}
\email{tiwaniec@syr.edu}
\thanks{Iwaniec was supported by the NSF grant DMS-0800416.}

\author{Leonid V. Kovalev}
\address{Department of Mathematics, Syracuse University, Syracuse,
NY 13244, USA}
\email{lvkovale@syr.edu}
\thanks{Kovalev was supported by the NSF grant DMS-0913474.}

\author{Jani Onninen}
\address{Department of Mathematics, Syracuse University, Syracuse,
NY 13244, USA}
\email{jkonnine@syr.edu}
\thanks{Onninen was supported by the NSF grant  DMS-0701059.}

\subjclass[2000]{Primary 31A05; Secondary 58E20, 30C20}

\date{March 15, 2009}

\keywords{Nitsche conjecture, Harmonic mappings}

\begin{abstract}
As long ago as 1962 Nitsche \cite{Ni} conjectured that a harmonic homeomorphism $h \colon A(r,R)  \overset{\textnormal{\tiny{onto}}}{\longrightarrow} A(r_\ast, R_\ast)$ between planar annuli exists if and only if
$ \frac{R_\ast}{r_\ast} \ge \frac{1}{2} \left(\frac{R}{r} + \frac{r}{R}\right)$.
We prove this conjecture when the domain annulus is not too wide; explicitly, when $\log \frac{R}{r} \le \frac{3}{2}$. For general $A(r,R)$ the conjecture is proved under additional assumption that either
$h$ or its normal derivative have vanishing average on the inner boundary circle.
This is the case for the critical Nitsche mapping which yields equality in the above inequality. The Nitsche mapping represents so-called free evolution of circles of the annulus $A(r,R)$. It will be shown on the other hand that forced harmonic evolution results in greater ratio $\frac{R_\ast}{r_\ast}$. To this end, we introduce the underlying differential operators for the circular means of the forced evolution and use them to obtain sharp lower bounds of $\frac{R_\ast}{r_\ast}$.
\end{abstract}
\maketitle
 \tableofcontents
\section{Introduction and Overview}
The Riemann Mapping Theorem tells us that planar simply connected domains (different from $\C$) are conformally equivalent. Annuli are the first place one meets obstructions to the existence of conformal mappings. The famous theorem, due to Schottky (1877), \cite{Sc} asserts that  an annulus
\begin{equation}
\mathbb A =   A (r,R)= \{z \in \mathbb C \colon r<|z|<R\}\, , \quad \quad 0<r<R
\end{equation}
can be mapped conformally onto the annulus
\begin{equation}
\mathbb A^\ast =   A (r_\ast,R_\ast)= \{w \in \mathbb C \colon r_\ast<|w|<R_\ast\}\, , \quad \quad 0<r_\ast<R_\ast
\end{equation}
if and only if $\Mod \A := \log \frac{R}{r} = \log \frac{R_\ast}{r_\ast}=:\Mod \A^\ast$; that is,
\begin{equation}
\frac{R}{r}= \frac{R_\ast}{r_\ast}.
\end{equation}
Moreover, modulo rotation, every conformal mapping $h \colon \mathbb A \to \mathbb A^\ast$ takes the form
\begin{equation}
h(z)= \frac{r^\ast}{r} z \quad \mbox{ or } \quad h(z) = \frac{r^\ast R}{z}.
\end{equation}
Note that the latter map, though orientation preserving, reverses the order of the boundary circles.
Such a mapping problem becomes more flexible if we  admit harmonic mappings $h \colon \mathbb A \overset{\textnormal{\tiny{onto}}}{\longrightarrow} \mathbb A^\ast$ in which the real and imaginary parts need not be harmonic conjugate. A univalent (one-to-one) complex-valued harmonic function will be referred to as  harmonic homeomorphism. We denote by $\mathcal H (\A, \A^\ast)$ the class of orientation preserving harmonic homeomorphisms
 $ h \colon {\mathbb A} \overset{\textnormal{\tiny{onto}}}{\longrightarrow} {\mathbb A}^\ast$ which preserve the order of the boundary circles, see Section~\ref{background} for a brief discussion of annuli and homeomorphisms between them.  For a recent account of the theory of harmonic mappings we refer to the book by P. Duren \cite{Dub}.

J. C. C. Nitsche \cite{Ni} showed that the annulus $\mathbb A$ cannot be mapped by a harmonic homeomorphism onto $\mathbb A^\ast$ if the target is conformally too thin compared to $\A$. Let us devote a few lines to  a simple proof of this fact via  normal family arguments. Suppose, to the contrary, that we are given harmonic homeomorphisms $h_j \colon \A \overset{\textnormal{\tiny{onto}}}{\longrightarrow} A(1,r_j)$, where $r_j \searrow 1$. Let $h \colon \A \to \C$ denote a $c$-uniform limit of this sequence. We have $\abs{h(z)}^2 \equiv 1$ on $\A$. Hence $0= \Delta \abs{h}^2 = 4 \frac{\partial^2}{\partial z \partial \bar z} (h\bar h) = \abs{h_z}^2 + \abs{h_{\bar z}}^2$. This implies that $h$ is constant. On the other hand,  we fix a circle
 \[C_{\rho}:=\{z\colon \abs{z}=\rho\} \quad \mbox{ where } r < \rho < R\]
and note that the winding number of $h_j$   equals  $1$, namely
  \[\frac{1}{2 \pi i \rho} \int_{C_{\rho}} \frac{1}{h_j} \frac{\partial h_j}{\partial \theta} =1.\]
  Passing to the limit as $r_j\searrow 1$,  we arrive at the desired contradiction. This proof does not provide us with any  bound for $\Mod \mathbb A^\ast$.
\begin{conjecture} [Nitsche~\cite{Ni}] An annulus $\mathbb A =  A(r,R)$ can be mapped by a harmonic homeomorphism onto the annulus $\mathbb A^\ast =   A (r_\ast,R_\ast)$ only if
\begin{equation}\label{nitschebound}
\frac{R^\ast}{r^\ast} \ge \frac{1}{2} \left( \frac{R}{r}+ \frac{r}{R}\right).
\end{equation}
We call it the {\it Nitsche bound}.
\end{conjecture}

Explicit lower  bounds of $\frac{R_\ast}{r_\ast}$ have been obtained by  Lyzzaik  \cite{Ly} (expressed in terms of the modulus of the Gr\"otzsch ring domain),   by  Weitsman   \cite{We}:
\[\frac{R_\ast}{r_\ast} \ge 1 + \frac{1}{2} \frac{r^2}{R^2} \log^2 \frac{R}{r}\]
and   by  Kalaj  \cite{Ka}:
\[\frac{R_\ast}{r_\ast} \ge 1 + \frac{1}{2}  \log^2 \frac{R}{r}.\]

\subsection{The Nitsche bound}
However, none of these lower bounds reach the critical number $\frac{1}{2} \left(\frac{R}{r}+ \frac{r}{R}\right)$.  In the present paper we prove~\eqref{nitschebound} when the domain annulus is conformally not too thick.
\begin{theorem}\label{thmain}
An annulus $\A= A(r,R)$ whose modulus $\Mod \A =\log \frac{R}{r} \le 3/2$   can be mapped harmonically onto the annulus $\A^\ast = A(r_\ast, R_\ast)$ if and only if
\[\frac{R_\ast}{r_\ast} \ge \frac{1}{2} \left(\frac{R}{r} + \frac{r}{R}\right). \]
\end{theorem}

The case when $\Mod (\A)>3/2$ remains open. Nevertheless  we  show that, regardless of the modulus of $\A$, the  Nitsche bound holds  for a class of harmonic homeomorphisms  $h \colon {\mathbb A} \overset{\textnormal{\tiny{onto}}}{\longrightarrow} {\mathbb A}^\ast$ having vanishing average on the inner circle of $\A$; that is,
\[\lim_{\rho \searrow r}\dashint_{C_{\rho}} h   =0.\]
Let us denote by $\mathcal H_D (\A , \A^\ast) \subset  \mathcal H (\A , \A^\ast)$ the subclass of  such  mappings.
Another class, denoted by $\mathcal H_N (\A , \A^\ast) \subset  \mathcal H (\A , \A^\ast)$, consists of harmonic homeomorphisms with vanishing normal derivative in the average; that is,
\[\dashint_{C_{\rho}}\frac{\D h}{\D n}=0 ,  \quad \mbox{for some (equivalently for all) } r< \rho < R.\]

\begin{theorem}\label{thzero}
Suppose $h \colon \A \to \A^\ast$ belongs to $\Ha_D (\A , \A^\ast)$ or $\Ha_N (\A , \A^\ast)$.
Then
\[\frac{R_\ast}{r_\ast} \ge \frac{1}{2} \left(\frac{R}{r} + \frac{r}{R}\right). \]
\end{theorem}

It turns out that the round shape of the outer boundary of the target annulus is not essential. From now on, to simplify matters, we normalize  $\A$ so that its inner boundary is the unit circle; that is, $\A=A(1,R)$. Moreover, the target $\As:= h(\A)$ will be   a half round annulus. Precisely $\As$  will be   a  doubly connected domain whose inner boundary is the unit circle.  The outer boundary of $\As$ can be arbitrary. The mean outer radius of $\mathcal A $ via the mapping $h \in\Ha   (\A, \mathcal A )$ is   defined by
\begin{equation}
R_\ast (h) = \lim_{\rho\nearrow R}\left(\dashint_{C_\rho} |h|^2\right)^\frac{1}{2}.
\end{equation}

Now Theorems \ref{thmain} and \ref{thzero} are special cases of the following results.
\begin{theorem}\label{thgene1}
Let $h \in\Ha  (\A, \As)$ with $\Mod \A  \le 3/2$. Then
\begin{equation}
R_\ast (h) \ge \frac{1}{2} \left(R + \frac{1}{R}\right).
\end{equation}
\end{theorem}
No restriction on the size of the annulus $\A$ will be imposed if  the mapping has vanishing average on the inner circle, or vanishing average of the normal derivative.

\begin{theorem}\label{thgene2}
Suppose $h$ belongs to one of the following
\begin{enumerate}[(i)]
\item\label{thgene21} $\Ha_D (\A , \As)$
\item\label{thgene22} $\Ha_N (\A , \As)$.
\end{enumerate}
Then
\begin{equation}
R_\ast (h) \ge \frac{1}{2} \left(R + \frac{1}{R}\right).
\end{equation}
\end{theorem}

Further related generalizations demand a few preliminary remarks.

\subsection{Extremal mappings}
Let us first look at the extremal harmonic mappings for the Dirichlet energy, see Section~\ref{secem} for details,
\begin{equation}\label{emf1}
h^\lambda (z)= \frac{1}{1+ \lambda} \left(z+ \frac{\lambda}{\overline{z}}\right) \colon \A  \to \A^\ast \, , \quad -1 < \lambda \le 1.
\end{equation}
The outer radius $R_\ast = \frac{R^2 + \lambda}{R+\lambda R}$ is smaller than $R$ if $\lambda >0$   and greater than $R$ if $-1 < \lambda <0$,  and $h^0$   is the identity mapping. These extremal mappings have vanishing average over the inner boundary, actually over any circle $C_\rho = \{z \colon |z|=\rho\} \subset \A$. Except for the critical case, corresponding to   $\lambda =1$, each $h^\lambda \colon \A \to \A^\ast$ has positive Jacobian determinant in the closure of $\A$.  Let us look more closely at    the {\it critical Nitsche mapping}
\[h^1 (z)= \frac{1}{2} \left(z + \frac{1}{\overline{z}}\right) = \frac{1}{2} \left(\rho + \frac{1}{\rho}\right) e^{i \theta}\, , \quad z \in \rho e^{i\theta}. \]
As opposed to other extremal mappings in ~\eqref{emf1}, the Jacobian determinant
\[J(z, h^1) = \left|h^1_z\right|^2 - \big|h^1_{\overline{z}}\big|^2 = \frac{|z|^4-1}{4|z|^4} \]
vanishes on  the inner boundary of $\A$. That is why perturbations of the boundary homeomorphisms $h^1 \colon C_1 \to C_1$ and $h^1 \colon C_R \to C_{R_\ast}$ might  destroy injectivity   inside $\A$. In fact,  we shall show that
\begin{theorem}\label{thmuni}(Uniqueness) For the critical configuration of annuli,  $R_\ast = \frac{1}{2}\left(R + \frac{1}{R}\right)$, we have
\[\Ha_D(\A, \A^\ast) = \Ha_N(\A, \A^\ast)= \{\alpha h^1 \colon \abs{\alpha}=1\}.\]
If in addition $\Mod \A \le 3/2$, then $\Ha(\A, \A^\ast)=  \{\alpha h^1 \colon \abs{\alpha}=1\}$.
\end{theorem}

 A natural  generalization of the Nitsche bound in Theorems~\ref{thmain}-\ref{thgene2} is possible   once we realize that the conjectured extremal mapping $h^1$ represents so-called free harmonic evolutions of the unit circle. To be precise, we regard $h\in \Ha (\A, \As)$ as a function of concentric circles $C_\rho =\{z \colon \abs{z}=\rho\} $, $1 < \rho < R$, into Jordan curves $C_\rho^\ast= h(C_\rho)$ in $\As$. We refer to
 \[\lim_{\rho \searrow 1} \frac{d}{d\rho} \left(\dashint_{C_\rho} \abs{h}^2 \right)^\frac{1}{2}   \ge 0 \]
 as the initial speed so that  free evolution begins with zero initial speed. As might be expected positive initial speed, or simply  forced harmonic evolution, results in larger ratio $R_\ast/r_\ast$. The extremal mappings $h^\lambda$, $-1< \lambda \le 1$, are representatives of the forced evolutions of circles. Their  initial speed is:
 \begin{equation}\label{inispeed}
\lim_{\rho \searrow 1} \frac{d}{d\rho} \left(\dashint_{C_\rho} \abs{h^\lambda}^2 \right)^\frac{1}{2}  = \frac{1-\lambda}{1+\lambda}.
 \end{equation}
\begin{theorem}\label{thnew}
Suppose $h\in H_D(\A, \As)$ has the initial speed
\[\lim_{\rho \searrow 1} \frac{d}{d\rho} \left(\dashint_{C_\rho} \abs{h}^2 \right)^\frac{1}{2} = \frac{1-\lambda}{1+\lambda} .\]
 Then
\[\left(\dashint_{C_\rho} \abs{h}^2 \right)^{1/2} \ge \left(\dashint_{C_\rho} \abs{h^\lambda}^2 \right)^{1/2}= \frac{\rho^2 +\lambda}{(1+\lambda) \rho}  \]
for $1<\rho < R$. Equality occurs if and only if $h= \alpha h^\lambda$ for some $\abs{\alpha} =1$.
\end{theorem}

 \section{Background}\label{background}

We shall work with the open annulus $\A = A(1,R)= \{z \colon 1 <\abs{z}<R\}$ whose inner boundary is the  unit circle $C_1= \{z \colon \abs{z}=1\}$. The basic complex harmonic functions in $\A$ are the integer powers $z^n$, $\bar z^n$, $n=0, \pm 1, \pm 2, \dots $, and the logarithm $\log \abs{z}$. If we combine these functions suitably in pairs, we obtain an orthogonal decomposition of a harmonic function,
\begin{equation}\label{hfeq1}
h(z)= \sum_{n\in \mathbb Z} h_n (z)
\end{equation}
where $h_n(z)= a_n z^n+b_n \bar z^{-n}$ for $n \ne 0$ and $h_0 (z)= a_0 \log \abs{z}+b_0$. For every circle $C_\rho = \{z \colon \abs{z}=\rho\}$, $1 \le \rho \le R$, we have
\begin{equation}\label{hfeq2}
\dashint_{C_\rho}h_n \bar h_m = \begin{cases} 0 & \mbox{ if } n \ne m \\
\abs{a_n \rho^n +b_n \rho^{-n}}^2 &  \mbox{ if } n = m \ne 0 \\
\abs{a_0 \log \rho  +b_0}^2 &  \mbox{ if } n = m = 0 .
\end{cases}
\end{equation}
It should be pointed out that $h$ need not be defined on the boundary of $\A$, but its components $h_n$ are well defined in the punctured complex plane $\C \setminus \{0\}$. The orthogonality of $h_n$, as well as their derivatives, will prove useful in the subsequent computations.

Given the rotational invariance of the annulus $\A$ and the radial symmetry of the extremal harmonic mappings, we shall express the complex variable $z=\rho\hspace{0.03cm} e^{i\theta}$ as a function of the polar coordinates $1< \rho <R$, $0 \le \theta < 2 \pi$. Then the Cauchy-Riemann derivatives are
\begin{equation}\label{cr}
\begin{split}
h_z&= \frac{\partial h}{\partial z}= \frac{1}{2}e^{-i\theta} \left(h_\rho - \frac{i}{\rho}h_\theta\right)\\
h_{\bar z}&= \frac{\partial h}{\partial \bar z}= \frac{1}{2}e^{i\theta} \left(h_\rho + \frac{i}{\rho}h_\theta\right).
\end{split}
\end{equation}
Hence the Laplacian
\begin{equation}\label{pceq3}
\Delta h = 4 \frac{\partial^2 h}{\partial z \partial \bar z} = \frac{1}{\rho} \frac{\partial}{\partial \rho} \left(\rho \frac{\partial h}{\partial \rho}\right) + \frac{1}{\rho^2} \frac{\partial^2 h}{\partial \theta^2}= h_{\rho \rho} + \frac{1}{\rho} h_\rho + \frac{1}{\rho^2} h_{\theta \theta}.
\end{equation}
We will use the Hilbert-Schmidt norm of the Jacobian matrix of $h$, which is equal to
\[\norm{Dh}^2 = 2 \left(\abs{h_z}^2 + \abs{h_{\bar z}}^2\right)= \abs{h_\rho}^2 + \rho^{-2} \abs{h_\theta}^2.\]
The Jacobian determinant is expressed in polar coordinates as
\[J(z,h)= \abs{h_z}^2 - \abs{h_{\bar z}}^2 = \frac{1}{\rho} \im \left( \bar h_\rho h_\theta \right).\]

We shall now fix a round annulus
\[\A = A(1,R) = \{z \colon 1 < \abs{z}<R\}, \quad 1 <R < \infty   \]
and consider homeomorphisms $h \colon \A  \overset{\textnormal{\tiny{onto}}}{\longrightarrow} \As$. Here $\As\subset\C$ is a topological annulus, i.e., an open connected set whose complement  $\widehat{\C} \setminus \As$ consists of two disjoint nonempty closed sets. One of these sets contains $\infty$ and is called the outer component, denoted $G_O$. The other set is the inner components, denoted $G_I$. Our standing assumption is $G_I=\{z\colon \abs{z}\le 1\}$.

In general $h$ does not extend continuously to the closure of $\A$, yet it ``takes'' the boundary circles
\[C_1 = \{z \colon \abs{z}=1\} \quad \mbox{and} \quad C_R = \{z \colon \abs{z}=R\} \]
into two different components of $\partial \As$ in the sense of cluster sets~\cite[p. 156]{Nebook}.
We write such correspondences as
\[h \colon C_1 \rightsquigarrow h\{C_1\} \quad \mbox{and} \quad h \colon C_R \rightsquigarrow h\{C_R\}.\]
There are four homotopy classes of homeomorphisms $h \colon \A  \overset{\textnormal{\tiny{onto}}}{\longrightarrow} \As$, each determined by the orientation of $h$ and the order of the boundary components $h\{C_1\}$ and $h\{C_R\}$ in $\partial \A$. We are interested in the following class.
\begin{definition}
The class $\Ho (\A , \As)$ consists of all orientation preserving homeomorphisms $h \colon \A  \overset{\textnormal{\tiny{onto}}}{\longrightarrow} \As$ whose cluster set at the inner circle $C_1 \subset \partial \A$ is the inner boundary of $\As$. The subclass of harmonic mappings in  $\Ho (\A , \As)$ will be designated by the symbol  $\Ha (\A , \As)$.
\end{definition}
Recall from introduction that   $h\in \Ho (\A , \As)$ is viewed as a function of circles $C_\rho = \{\rho e^{i \theta} \colon 0 \le \theta < 2 \pi\}$ into Jordan curve $C_\rho^\ast=h(C_\rho)$, $1<\rho <R$. This function is called the {\it evolution of circles}. For the generalization of the Nitsche conjecture it will be essential to assume that the evolution begins with $h \colon C_1 \rightsquigarrow C_1$.  We then note that for $h \in \Ho (\A , \As)$ the function $z \to |h(z)|$ extends continuously to the inner circle of $\A$, with value $1$ therein.

\section{Circular means of a harmonic evolution}\label{seccm}
We shall now introduce a number of integral means over the circles $C_\rho$, $1<\rho<R$, first defined for harmonic functions $h \colon \A \to \C$ and then restricted to the mappings in $\Ha (\A , \As)$. Our aim is to indicate in some detail how to understand these integral means on the inner circle $C_1$, as  $h$ is not even defined on $C_1$.  The orthogonal decomposition~\eqref{hfeq1} of $h$ on $\A$ comes handful. Accordingly, the circular means
\begin{equation}\label{cm1}
\dashint_{C_\rho} h = \frac{1}{2 \pi \rho} \int_{C_\rho} h = a_0 \log \rho +b_0 , \quad 1 <\rho <R,
\end{equation}
extend continuously to the closed interval $[1,R]$,
\begin{equation}\label{cm2}
 \lim_{\rho \searrow 1} \dashint_{C_\rho} h =b_0.
\end{equation}
Their $\rho$-derivatives can also be given a meaning at the boundary circles,
\begin{equation}\label{cm3}
  \lim_{\rho \searrow 1} \frac{d}{d \rho} \dashint_{C_\rho} h =  \lim_{\rho \searrow 1}   \dashint_{C_\rho} h_\rho= a_0.
\end{equation}
Hereafter, passing with differentiation inside the integral is justified by a general commutation rule, which in symbols reads as
\begin{equation}\label{cm4}
 \frac{d}{d \rho} \dashint_{C_\rho} =  \dashint_{C_\rho} \frac{d}{d \rho}, \quad \mbox{ for } 1 < \rho <R.
\end{equation}
Next we introduce
\begin{equation}\label{cm5}
U=U(\rho)=U(\rho,h)= \dashint_{C_\rho}\abs{h}^2, \quad 1 < \rho <R
\end{equation}
and call $U(\rho)$ the {\it quadratic mean} of $h$ over the circle $C_\rho$. Using the Green's formula $\iint_{\Omega} \left(u \Delta v -v \Delta u\right) = \int_{\partial \Omega} \left(u \frac{\partial v}{\partial n}- v  \frac{\partial u}{\partial n}\right)$, with $u=1$ and $v=|h|^2$, this leads to
\begin{eqnarray}\label{cm6}
\int_{C_\rho} \abs{h^2}_\rho - \int_{C_r} \abs{h^2}_\rho &=& \iint_{r \le \abs{z} \le \rho} \Delta \abs{h}^2 = 4 \iint_{r \le \abs{z} \le \rho} \frac{\partial^2 \abs{h}^2}{\partial z \partial \bar z} \nonumber \\
&=&4  \iint_{r \le \abs{z} \le \rho} \left(\abs{h_z}^2 + \abs{h_{\bar z}}^2 \right)= 2  \iint_{r \le \abs{z} \le \rho} \norm{Dh}^2 .
\end{eqnarray}
Here $ 1 <r \le \rho <R$. By virtue of the commutation rule~\eqref{cm4} we obtain
\begin{equation}\label{cm7}
\rho \,  \dot{U}(\rho)- r \, \dot{U}(r) = 2 \iint_{r \le \abs{z} \le \rho} \norm{Dh}^2 , \quad 1 <r \le \rho <R
\end{equation}
where, as usual, dot over $U$ stands for the $\rho$-derivative of $U$. Differentiation of~\eqref{cm7} with respect to $\rho$ yields
\begin{equation}\label{cm8}
\frac{1}{\rho} \frac{d}{d \rho} \left[\rho \left( \frac{d}{d\rho} U \right) \right]= 4\pi \dashint_{C_\rho} \norm{Dh}^2>0 , \quad 1 < \rho <R.
\end{equation}
It should be said, therefore, that $U$ is a subsolution to the differential operator
\begin{equation}\label{cm9}
L=\frac{1}{\rho} \frac{d}{d \rho} \left[\rho   \frac{d}{d\rho}  \right] , \quad L[U] \ge 0 \mbox{ in } (1,R).
\end{equation}
We shall see in Section~\ref{secuc} many more second order differential operators to which various integral means are subsolutions. As regards the extension of $U(\rho)$ and its derivative $\dot{U}(\rho)$ to $[1,R)$, we restrict ourselves to the mappings $h\in \Ha (\A , \As)$, so that the function $z\to |h(z)|$ is continuous up to the inner circle of $\A$. In particular, setting $U(1)=1$ gives the desired continuous extension of $U(\rho)$ to all radii $1 \le \rho <R$. Then formula~\eqref{cm8} allows us to extend $\dot{U}(\rho)$ to $[1,R)$. To this end, we infer from~\eqref{cm8} that $\rho \,  \dot{U}(\rho)$ is strictly increasing. Hence
\[\dot{U}(\rho)= \frac{1}{\rho \log \rho} \int_1^\rho \rho \,  \dot{U}(\rho) \frac{\dtext t}{t} > \frac{1}{\rho \log \rho} \int_1^\rho t \, \dot{U}(t) \frac{\dtext t}{t}  = \frac{U(\rho)-U(1)}{\rho \log \rho}>0 .\]
We then conclude that the following limit exists, and is nonnegative.
\begin{equation}\label{uprime}
\lim_{\rho \searrow 1} \dot{U} (\rho) = \lim_{\rho \searrow 1} \rho\,  \dot{U} (\rho) =: \dot{U}(1) \ge 0.
\end{equation}
Now the function $\dot{U}(\rho)$, so defined at $\rho =1$, is clearly continuous on $[1,R)$. Moreover $\dot U (1)$ agrees with the usual definition of the derivative,
\[\frac{U(\rho)-1}{\rho-1} = \frac{1}{\rho-1} \int_1^\rho \dot{U}(t)\,  \dtext t \to \dot{U}(1) \quad  \mbox{ as } \rho \searrow 1.\]
Formula~\eqref{cm7} now remains valid for $r=1$,
\[2 \iint_{1 \le \abs{z} \le \rho} \norm{Dh}^2 = \rho \, \dot{U}(\rho)- \dot{U}(1) \le \rho \,  \dot{U} (\rho).\]
Hence
\begin{equation}\label{cm10}
\dot{U}(\rho) \ge \frac{2}{\rho} \iint_{1 \le \abs{z} \le \rho} \norm{Dh}^2>0 \quad \mbox{ for } 1 < \rho <R
\end{equation}
and, as a consequence, we infer that
\begin{proposition} Suppose $h\in \Ha(\A,\As)$.
Then the quadratic means $U(\rho)= \dashint_{C_\rho} |h|^2$ are strictly increasing and the integral of $\norm{Dh}^2$
over $A(1,\rho)$ is finite for $1<\rho<R$.
\end{proposition}
It is natural to define
\begin{equation}\label{UR}
U(R):=\lim_{\rho\nearrow R}U(r)=(R_*(h))^2\in [0,+\infty].
\end{equation}

When $\rho$ approaches the outer radius of $\A$, the integral $\iint_{A(1,\rho)}\norm{Dh}^2$
may grow to infinity. A short sketch proof of this fact runs somewhat as follows:

\begin{example}
Let $f \colon C_1 \to C_1$ be a given homeomorphism of the unit circle onto itself. We consider the Poisson integral extension of $f$ into the unit disk $\mathbb D$, still denoted by $f$. This extension is a homeomorphism of $\overline{ \mathbb D}$ onto itself by  the Rad\'o-Kneser-Choquet Theorem, and $C^\infty$-smooth diffeomorphism in the open disk by Lewy's Theorem,  see~\cite{Dub}. In general, the Dirichlet energy of $f$ need not be finite. The best that one can guarantee is that $\norm{Dh}$ lies in the Marcinkiewicz space $L^2_{\textnormal{weak}} (\mathbb D)$, see~\cite{IMS}. Having disposed of such a harmonic mapping $f \colon \mathbb D \to \mathbb D$ with infinite energy, we look at the inverse image of an annulus $A=A(r,1) \subset \mathbb D$, $0<r<1$, to observe that $f^{-1}(A)$ is a doubly connected domain with smooth boundaries (real analytic). There exists a conformal mapping $\varphi \colon \A  \overset{\textnormal{\tiny{onto}}}{\longrightarrow} f^{-1}(A)$ of a round annulus $\A=A(1,R)$ onto $f^{-1}(A)$. This mapping is a diffeomorphism up to the boundary of $\A$, even in a neighborhood of $\overline{\A}$. In this way we arrive at the harmonic homeomorphism with infinite energy
\[h= \frac{1}{r} (f \circ \varphi)  \colon A(1,R) \to A(1,R_\ast) , \quad R_\ast = \frac{1}{r}.\]
\end{example}
We turn next to the introduction of the variance of $h$. Observe that the circular means $\dashint_{C_\rho} h = a_0 \log \abs{z}+b_0$ form a harmonic function, so is the function $H=h- \dashint_{C_\rho} h $  whose quadratic average is the variance of $h$,
\[V=V(\rho)=V(\rho,h) = \dashint_{C_\rho} \left|h-  \dashint_{C_\rho}h  \right|^2 = \dashint_{C_\rho} |h|^2 - \left| \dashint_{C_\rho} h \right|^2 = U(\rho, H).\]
The orthogonal decomposition~\eqref{hfeq1} gives rise to a decomposition of the circular means, namely
\begin{equation}\label{cm11}
U=U(\rho,h)=\sum_{n\in \Z}U_n (\rho) , \quad U_n(\rho)=U(\rho, h_n)
\end{equation}
\begin{equation}\label{cm12}
V=V(\rho,h)=\sum_{n\ne 0}U_n (\rho) .
\end{equation}
We shall explore these formulas   throughout this paper.

At this stage we take advantage of the orthogonal decomposition  to deduce that $V(\rho)$ is convex in $\rho$.
\[V(\rho)= \sum_{n \ne 0} \left| a_n \rho^n + b_n \rho^{-n} \right|^2 =  \sum_{n \ne 0} \left(\abs{a_n}^2 \rho^{2n} + \abs{b_n}^2 \rho^{-2n} + 2 \re a_n \bar b_n  \right) .\]
The second derivative is indeed positive
\begin{equation}\label{second}
\ddot{V}(\rho)= \frac{2}{\rho^2} \sum_{n\in \Z} \left[n(2n-1) \abs{a_n}^2 \rho^{2n} + n(2n+1) \abs{b_n}^2 \rho^{-2n}\right]>0.\end{equation}
Thus, $V$ is a subsolution to the operator $\frac{d^2}{d \rho^2}$. However, this operator is not good enough for the conclusion of the Nitsche conjecture, because the critical Nitsche mapping fails to be a solution.

\section{The extremal mappings}\label{secem}
Let us look more closely at the   minimizers of  the Dirichlet integral
\begin{equation}\label{em1}
\mathcal E[h] = \iint_{\A} \norm{Dh}^2
\end{equation}
subject to all homeomorphisms $h\in \Ho(\A, \A^\ast)$, between round annuli $\A=A(1,R)$ and $\A^\ast=A(1,R_\ast)$. We invoke the results in \cite{AIM, IO}. Within the Nitsche range~\eqref{nitschebound} for the annuli $\A$ and $\A^\ast$ the minimum is obtained (uniquely up to rotation) by the harmonic mapping
\begin{equation}\label{em2}
h^\lambda (z)= \frac{1}{1+\lambda} \left(z + \frac{\lambda}{\bar z}\right) \quad \mbox{where }  \lambda = \frac{R^2-RR_\ast}{RR_\ast-1} \in (-1,1].
\end{equation}
Outside the Nitsche range~\eqref{nitschebound} for the annuli $\A$ and $\A^\ast$ the infimum of $\mathcal H (\A,\A^\ast)$ is not attained~\cite{AIM, IO}. It is a general fact concerning mappings between domains in $\C$ that any minimizer of the Dirichlet energy is harmonic outside the branch set. The reader may wish to know that the nonharmonic mapping $h(z)= z/|z|$ of $\A$ onto the unit circle is a minimizer of the Dirichlet integral~\cite{IO}. We then see that the Nitsche conjecture, if true, would imply nonexistence of minimizers outside of the Nitsche bound but not vice versa.

 If one thinks of $h^\lambda$ as evolution of circles in $A(1,R)$, then the parameter $\lambda$, $-1< \lambda \le 1$, tells us about the initial speed of the evolution. Let us denote by $r_\ast (\rho, h^\lambda)$  the radii of the circles $C_\rho^\ast=h^\lambda (C_\rho)$. It is easily seen that
\[\dot{r}_\ast (1,h^\lambda)= \frac{d}{d \rho} r_\ast (\rho, h^\lambda) \Big|_{\rho=1} = \frac{1-\lambda}{1+\lambda}\in [0, \infty).\]
This observation will be useful in Section~\ref{aftervar} where we show that  $h^\lambda$ also serves as extremal among all   harmonic evolutions $h\in \Ha_D(\A, \As)$ of the given initial speed.

\section{Convexity operators}\label{secuc}
The idea is to generalize the usual convexity operator $\frac{d^2}{d \rho^2}$ in such a way that variance $V=V(\rho)= V(\rho, h)$ of all complex harmonic functions would satisfy the inequality
\[\LL [V]:= \ddot{V} + A(\rho) V + B(\rho)V \ge 0\]
with equality occurring exactly for the extremal mapping $h^\lambda = \frac{1}{1+\lambda} \left(z + \frac{\lambda }{\bar z}\right)$, $-1 < \lambda \le 1$. This goal will be carried uniquely through the following operator
\begin{equation}\label{uc1}
\LL^\lambda = \frac{d^2}{d \rho^2} + \frac{3 \lambda - \rho^2}{\rho (\rho^2 +\lambda)} \frac{d}{d \rho} - \frac{8\lambda}{(\rho^2+\lambda)^2}
\end{equation}
or, in divergence form
\begin{equation}\label{uc2}
\LL^\lambda [V] = \frac{\rho^2 + \lambda}{\rho^3} \frac{d}{d \rho} \left[ \rho^3 \frac{d}{d \rho} \left(\frac{V}{\rho^2 + \lambda}\right) \right]
\end{equation}
which is defined for any $V\in C^2 (1,R)$. Before proving the above property of $\LL^\lambda$, see Proposition~\ref{varprop1},  two other identities are worth noting.
\begin{lemma}\label{uclem1}
Let $h$ be a complex harmonic function on $\A=A(1,R)$ and $U=U(\rho)= \dashint_{C_\rho}|h|^2$. Then
\begin{equation}\label{uc3}
\LL^\lambda [U]=2  \dashint_{C_{\rho}}\left[\norm{Dh}^2
-\frac{1}{\rho}\frac{d}{d\rho}\left(\frac{\rho^2-\lambda}{\rho^2+\lambda}\abs{h}^2\right)\right]
\end{equation}
\begin{equation}\label{uc4}
\LL^\lambda [U]= \frac{2}{\rho^2}\,\dashint_{C_{\rho}}\left[\abs{h_{\theta}}^2-\abs{h}^2
+(\rho^2+\lambda)^2\left|\frac{d}{d\rho}\left(\frac{\rho}{\rho^2+\lambda}h\right)\right|^2\right].
\end{equation}
\end{lemma}
\begin{proof}
In view of the commutation rule~\eqref{cm4} we find that
differentiation yields
\[
\dot{U}(\rho)=  \dashint_{C_\rho} \abs{h^2}_\rho=2 \dashint_{C_\rho} \re \left(\bar h h_\rho \right)
\]
and
\[
\ddot{U}(\rho)= 2 \dashint_{C_\rho}\left[\, \abs{h_{\rho}}^2+\re(\bar h h_{\rho\rho})\, \right].
\]
Since $h$ is harmonic, we have the Laplace equation $h_{\rho \rho}=-h_\rho/\rho-h_{\theta \theta}/\rho^2$. Integrating $\bar h h_{\theta \theta}$ by parts over the circle $C_\rho$ yields
\[\ddot{U}(\rho) = \dashint_{C_\rho} \left[ 2 \abs{h_\rho}^2 - \frac{1}{\rho} \abs{h^2}_\rho + \frac{2}{\rho^2}\abs{h_\theta}^2  \right].\]
Substitute these values of $U$, $\dot{U}$ and $\ddot{U}$ into~\eqref{uc1} to obtain
\[\LL^\lambda [U] = \dashint_{C_\rho} \left[ 2 \abs{h_\rho}^2 + \frac{2}{\rho^2} \abs{h_\theta}^2 -2 \frac{\rho^2-\lambda}{\rho (\rho^2+\lambda)} \abs{h^2}_\rho - \frac{8 \lambda}{(\rho^2+\lambda)^2} \abs{h}^2  \right].\]
One way to continue this computation  is to express $\norm{Dh}^2$ in polar coordinates $\norm{Dh}^2 = \abs{h_\rho}^2 + \rho^{-2} \abs{h_\theta}^2$ and group the remaining terms to arrive at formula~\eqref{uc3}. For   formula~\eqref{uc4}, however, we proceed  in the following way
\begin{equation}
\begin{split}
\LL^{\lambda}[U]&=\frac{2}{\rho^2}\,\dashint_{C_\rho}\left\{\abs{h_{\theta}}^2+\rho^2\abs{h_{\rho}}^2
-\frac{\rho(\rho^2-\lambda)}{(\rho^2+\lambda)}\abs{h^2}_{\rho}-\frac{4\lambda\rho^2}{(\rho^2+\lambda)^2}\abs{h}^2\right\} \nonumber \\
&=\frac{2}{\rho^2}\,\dashint_{C_{\rho}}\left\{\abs{h_{\theta}}^2-\abs{h}^2
+ \left[\frac{(\rho^2-\lambda)^2}{(\rho^2+\lambda)^2} \abs{h}^2+ \rho^2 \abs{h_\rho}^2 - \frac{\rho (\rho^2 -\lambda)}{\rho^2 +\lambda} \abs{h^2}_\rho  \right]  \right\} .\nonumber
\end{split}
\end{equation}
It only remains to verify that the expression in the rectangular parentheses coincides with the term
\begin{equation}
\begin{split}
& (\rho^2+\lambda)^2 \left| \frac{d}{d\rho} \left( \frac{\rho h}{\rho^2 +\lambda} \right) \right|^2 = (\rho^2+\lambda)^2 \left| \frac{\rho h_\rho}{\rho^2 +\lambda}- \frac{\rho^2 -\lambda}{(\rho^2 +\lambda)^2}h \right|^2\\
&= (\rho^2+\lambda)^2 \Bigg[ \frac{\rho^2 \abs{h_\rho}^2}{(\rho^2 +\lambda)^2} + \frac{(\rho^2 -\lambda)^2}{(\rho^2+\lambda)^4}|h|^2    -2 \frac{\rho (\rho^2-\lambda)}{(\rho^2+\lambda)^3} \re \bar h h_\rho  \Bigg] \\
&= \frac{(\rho^2-\lambda)^2}{(\rho^2 + \lambda)^2} + \rho^2 \abs{h_\rho}^2 -  \frac{\rho (\rho^2-\lambda)}{\rho^2 + \lambda } \abs{h^2}_\rho
\end{split}
\end{equation}
as desired.
\end{proof}
Two special cases are worth noting. First, the operator $\LL^1$ for the critical Nitsche mapping $h^1= \frac{1}{2} \left(z+\frac{1}{\bar z}\right)$  takes the form
\[\LL^1 = \frac{d^2}{d \rho^2} + \frac{3-\rho^2}{\rho (\rho^2 +1)} \frac{d}{d \rho} - \frac{8}{(\rho^2+1)^2}, \quad \LL^1[U]=0 \quad \mbox{with } U= \frac{(\rho^2+1)^2}{4 \rho^2}.\]
On the other hand, letting $\lambda =0$, we obtain the operator associated with the identity mapping $h^0(z)=z$,
\[\LL^0= \frac{d^2}{d \rho^2} - \frac{1}{\rho} \frac{d}{d\rho}, \quad  \LL^0 [U]=0 \quad \mbox{with } U=\rho^2.\] For $\lambda =1$ the critical Nitsche mapping $h^1$ represents harmonic mapping $h\in \Ha (\A, \As)$ with vanishing normal derivative on the inner circle. Such mappings extend harmonically beyond the unit circle by reflection
\[h(z)= h\left(1/ \bar z\right) \quad \quad \mbox{ for } z\in A(1/R, 1).\]
The extended mapping is a double cover of $\As$. For $\lambda =0$, on the other hand, the identity mapping $h^0(z)=z$ represents, in particular, all conformal evolutions of the unit circle.
\begin{lemma}\label{uclem2}
Every evolution of circles in $\A$ generated by a conformal mapping $h\in \Ha (\A, \As)$ begins with the unit speed; that is,
\[\dot{r}_\ast (1,h)= \lim_{\rho \searrow 1} \frac{d}{d \rho} \left( \dashint_{C_\rho} \abs{h}^2\right)^{1/2}  =1.\]
\end{lemma}
\begin{proof}
First, we extend $h$ conformally to the annulus $A(1/R,R)$, by reflection
\[h(z)= \overline{[h(1/\bar z)]^{-1}} , \quad \mbox{ for } \frac{1}{R}< \abs{z} \le 1.\]
The Cauchy-Riemann system $h_{\bar z}=0$ reads in polar coordinates as $h_\rho = \frac{-i}{\rho}h_\theta$, and  the winding number of $h$ equals 1.  Hence,
\[\dashint_{C_\rho} \frac{-i h_\theta}{h}=1, \quad \quad \mbox{for every } \frac{1}{R} < \rho < R.\]
Now   the computation of the initial speed proceeds as follows.
\begin{equation}
\begin{split}
\dot{r}_\ast (1,h)&=   \frac{1}{2} \dashint_{C_1} \abs{h^2}_\rho =  \frac{1}{2} \dashint_{C_1} \left( h_\rho \bar h + \bar h_\rho h \right)\\
&=  \frac{1}{2} \dashint_{C_1} \frac{-i h_\theta}{h} +  \frac{1}{2} \overline{ \dashint_{C_1}  \frac{-i h_\theta}{h} }=1
\end{split}
\end{equation}
as claimed.
\end{proof}

\section{Variance is a subsolution to all $\LL^{\lambda}$, $-1<\lambda\le 1$.}\label{secvs}

The quadratic means of a harmonic function $H=h-\dashint_{C_{\rho}}h$  are none other than   the variance of $h$ which  can be computed by using
orthogonal decomposition~\eqref{hfeq1}.
\begin{equation}\label{var1}
V(\rho,h)=U(\rho,H)=\sum_{n\ne 0}U(\rho,h_n).
\end{equation}
We employ formula~\eqref{uc4} for $\LL^{\lambda}$. Neglecting the nonnegative term
\begin{equation}\label{var2}
\left|\frac{d}{d\rho}\left(\frac{\rho}{\rho^2+\lambda}H\right)\right|^2\ge 0
\end{equation}
yields the desired inequality
\begin{equation}\begin{split}\label{var3}
\LL^{\lambda}[V]&\ge\frac{2}{\rho^2}\dashint_{C_{\rho}}(\abs{H_{\theta}}^2-\abs{H}^2) \\
&=\frac{2}{\rho^2}\sum_{n\ne 0}\dashint_{C_{\rho}}\left(\left|\frac{d}{d\theta}h_n\right|^2-\abs{h_n}^2\right) \\
&=\frac{2}{\rho^2}\sum_{n\ne 0}(n^2-1)U(\rho,h_n)\ge 0.
\end{split}\end{equation}

We are now  interested in when  equality  $\LL^{\lambda}[V]= 0$ occurs. For this, we must have equality at~\eqref{var2},
which yields
\begin{equation*}
H(\rho e^{i\theta})=\frac{1}{1+\lambda}\left(\rho+\frac{\lambda}{\rho}\right)C(\theta),
\end{equation*}
where $C=C(\theta)$ is a $2\pi$-periodic function in $0\le\theta\le 2\pi$. As $H$ is harmonic, using the Laplace equation
in polar coordinates, we find that
\begin{equation*}\begin{split}
0&=\Delta H=H_{\rho\rho}+\frac{1}{\rho}H_{\rho}+\frac{1}{\rho^2}H_{\theta\theta} \\
&=\frac{\rho^2+\lambda}{(1+\lambda)\rho^3}[C(\theta)+C''(\theta)].
\end{split}\end{equation*}
The general solution to this ODE is $C(\theta)=\alpha \, e^{i\theta}+\beta \, e^{-i\theta}$, so
\begin{equation*}
H(z)=\frac{\alpha}{1+\lambda}\left(z+\frac{\lambda}{\bar z}\right)+\frac{\beta}{1+\lambda}\left(\bar z+\frac{\lambda}{z}\right)
=\alpha\,  h^{\lambda}(z)+\beta \, \overline{h^{\lambda}}(z),\quad \alpha,\beta\in\C.
\end{equation*}
Hence
\begin{equation}\label{var4}
h(z)=\frac{\alpha}{1+\lambda}\left(z+\frac{\lambda}{\bar z}\right)+\frac{\beta}{1+\lambda}\left(\bar z+\frac{\lambda}{z}\right)
+a_0\log \abs{z}+b_0.
\end{equation}
Of additional interest to us is the case when $h\in \Ha(\A,\As)$. This $h$, given by~\eqref{var4}, has modulus $1$ on $C_1$.
In polar coordinates,
\begin{equation}\label{var5}
h(e^{i\theta})=\alpha e^{i\theta}+\beta e^{-i\theta}+b_0.
\end{equation}
It is easily seen that the condition $\abs{h(e^{i\theta})}^2\equiv 1$ yields
\begin{equation}\label{var6}
\abs{\alpha}^2+\abs{\beta}^2+\abs{b_0}^2=1,\quad \beta\bar\alpha=0, \quad b_0\bar\beta+\bar b_0\alpha=0.
\end{equation}
The possibility $\alpha=\beta=0$ is ruled out because $h(e^{i\theta})\not\equiv\const$. This leaves two cases:

\underline{Case 1}, $\alpha=0$ and $\beta\ne0$. The equations~\eqref{var6} reduce to $\abs{\beta}=1$ and $b_0=0$, so we obtain
\begin{equation*}
h(z)=\frac{\beta}{1+\lambda}\left(\bar z+\frac{\lambda}{z}\right)+a_0\log\abs{z}.
\end{equation*}
Next we look at how this mapping stands up to the circular mean of the Jacobian determinant $J(z,h)=\abs{h_z}^2-\abs{h_{\bar z}}^2$,
where
\begin{equation*}
h_z=-\frac{\beta\lambda}{(1+\lambda)z^2}+\frac{a_0}{2z}\quad
\text{and} \quad h_{\bar z}=\frac{\beta}{1+\lambda}+\frac{a_0}{2\bar z}.
\end{equation*}
Using orthogonality of the power functions on circles, we obtain a contradiction
\begin{equation*}
0\le \dashint_{C_{\rho}}J(z,h)=\frac{\abs{\beta}^2}{(1+\lambda)^2}\left(\frac{\lambda^2}{\rho^4}-1\right)<0
\quad \text{for }\rho>1.
\end{equation*}
Therefore, the only possibility is:

\underline{Case 2}: $\beta=0$ and $\alpha\ne 0$. As before, the mapping $h$ takes the form
\begin{equation*}
h(z)=\frac{\alpha}{1+\lambda}\left(z+\frac{\lambda}{\bar z}\right)+a_0\log\abs{z}.
\end{equation*}
We just proved the following

\begin{proposition}\label{varprop1}
Variance $V=V(\rho)$ of a harmonic function $h\colon A(1,R)\to\C$ is a subsolution to all operators
$\LL^{\lambda}$ with $-1<\lambda\le 1$; that is, $\LL^{\lambda}[V]\ge 0$. Equality $\LL^{\lambda}[V]\equiv 0$,
for some $\lambda$, occurs if and only if
\begin{equation}\label{var7}
h(z)=b_0+a_0\log\abs{z}+\frac{\alpha}{1+\lambda}\left(z+\frac{\lambda}{\bar z}\right)+\frac{\beta}{1+\lambda}\left(\bar z+\frac{\lambda}{z}\right)
\end{equation}
where $a_0,b_0,\alpha,\beta$ are arbitrary complex coefficients. If, moreover, $h\in\Ha(\A,\As)$, then
 \begin{equation}\label{var8}
h(z)=a_0\log\abs{z}+\frac{\alpha}{1+\lambda}\left(z+\frac{\lambda}{\bar z}\right),\quad  \mbox{with }\abs{\alpha}=1.
\end{equation}
Here the coefficient  $a_0$ must be small enough so that $h$ is injective in the entire annulus $\A$.
\end{proposition}

The reader may wish to notice that the equality
$\LL^{\lambda}[V]=0$ for $h\in\Ha (\A,\As)$ implies $\lim\limits_{\rho \searrow 1}\dashint_{C_\rho}h=0$. Here is yet further reduction of formula~\eqref{var8}
for mappings between round annuli, meaning that $\abs{h(z)}=R_*$ for $\abs{z}=R$.
In this case equality $\LL^{\lambda}[V]\equiv 0$ occurs only for
\begin{equation*}
h(z)=\frac{\alpha}{1+\lambda}\left(z+\frac{\lambda}{\bar z}\right)=\alpha h^{\lambda}(z),\quad \abs{\alpha}=1,
\end{equation*}
which is, up to rotation, the extremal mapping for the Dirichlet energy $\iint_{\A}\norm{Dh}^2$. We have in this case
$\dashint_{C_{\rho}}h=0$ for every $1<\rho<R$.

\section{Proof of Theorems~\ref{thgene2}~\eqref{thgene21}  and~\ref{thnew}}\label{aftervar}

Theorem~\ref{thnew} is a special case of the following result.

\begin{proposition}\label{mean0}
Consider an arbitrary harmonic function $h\colon A(1,R)\to\C$ normalized at the inner circle by the conditions
\begin{equation}\label{aftervar1}
\lim_{\rho\searrow 1}\dashint_{C_{\rho}}h=0\quad \text{and }\quad \lim_{\rho\searrow 1}\dashint_{C_{\rho}}\abs{h}^2=1.
\end{equation}
Suppose that the evolution of circles under $h$ begins with nonnegative speed
\begin{equation}\label{aftervar2}
\lim_{\rho\searrow 1} \frac{d}{d\rho}\left(\dashint_{C_{\rho}}\abs{h}^2\right)^{1/2}=\frac{1-\lambda}{1+\lambda}\ge 0,
\end{equation}
where $-1<\lambda\le 1$. Then
\begin{equation}\label{aftervar3}
\left(\dashint_{C_{s}}\abs{h}^2\right)^{1/2}\ge \frac{s^2+\lambda}{(1+\lambda)s}\ge\frac{1}{2}\left(s+\frac{1}{s}\right)
\end{equation}
for $1<s<R$. The first inequality in~\eqref{aftervar3} turns into  equality if and only if $h=\alpha h^\lambda$  with $\abs{\alpha}=1$.
\end{proposition}

\begin{proof}  Let us examine the variance
\begin{equation*}
V=V(\rho)=\dashint_{C_{\rho}}\abs{h}^2-\left|\dashint_{C_{\rho}}h\right|^2.
\end{equation*}
A key step in obtaining~\eqref{aftervar3} is the inequality from Proposition~\ref{varprop1},
\begin{equation}\label{u0}
0\le \LL^{\lambda}[V]=\frac{\rho^2+\lambda}{\rho^3}\,\frac{d}{d\rho}\left[\rho^3\frac{d}{d\rho}\left(\frac{V}{\rho^2+\lambda}\right) \right]
\end{equation}
which tells us that the function $\rho\mapsto \rho^3\frac{d}{d\rho}\left(\frac{V}{\rho^2+\lambda}\right)$ is nondecreasing. An obvious consequence
of it is that
\begin{equation}\label{aftervar4}
\frac{d}{d\rho}\left(\frac{V}{\rho^2+\lambda}\right)\ge\frac{C}{\rho^3},\quad 1\le\rho<R,
\end{equation}
where the constant $C$ is given by
\begin{equation}\label{aftervar5}
C=\frac{d}{d\rho}\left(\frac{V}{\rho^2+\lambda}\right)\bigg|_{\rho=1}=\frac{\dot V(1)}{1+\lambda}-\frac{2V(1)}{(1+\lambda)^2}.
\end{equation}
We express $C$ in terms of the initial speed. Using the normalization conditions in~\eqref{aftervar1} gives $V(1)=1$.
More generally,
\begin{equation*}
V(\rho)=\dashint_{C_{\rho}}\abs{h}^2-\left|\dashint_{C_{\rho}}h\right|^2=\dashint_{C_{\rho}}\abs{h}^2-\abs{a_0}^2\log^2\rho,
\end{equation*}
where we employed the orthogonal decomposition~\eqref{hfeq1} with $b_0=\dashint_{C_1}h=0$. Differentiation yields
\begin{equation*}
\dot V(1)=\dot{U}(1)   
=2\frac{1-\lambda}{1+\lambda}.
\end{equation*}
Substitute these values of $V(1)$ and $\dot V(1)$ into~\eqref{aftervar5} to find the constant $C$
\begin{equation}\label{aftervar6}
C=-\frac{2\lambda}{(1+\lambda)^2}.
\end{equation}
Then inequality~\eqref{aftervar4} takes the explicit form
\begin{equation}\label{u1}
\frac{d}{d\rho}\left(\frac{V}{\rho^2+\lambda}\right)\ge -\frac{2\lambda}{(1+\lambda)^3}\frac{1}{\rho^3}.
\end{equation}
We integrate it over the interval $1<\rho<s$ to obtain
\begin{equation}\label{u2}
V(s)\ge \frac{(s^2+\lambda)^2}{(1+\lambda)^2s^2}
\end{equation}
which yields the desired inequality,
\begin{equation}\label{blah}
 \dashint_{C_s}\abs{h}^2=V(s)+\left|\dashint_{C_s}h\right|^2\ge \left[\frac{s^2+\lambda}{(1+\lambda)s}\right]^2
 \end{equation}
Now suppose that the equality occurs in~\eqref{blah}. Then  $\LL^\lambda [V] \equiv 0$ and $\dashint_{C_s}h=0$. The latter together with~\eqref{aftervar1} yield $a_0=b_0=0$. The equality case of Proposition~\ref{varprop1} implies $h=\alpha h^\lambda$ where $\abs{\alpha}=1$.
\end{proof}

\begin{proof}[Proof of Theorem~\ref{thgene2}~\eqref{thgene21}]
Choose $\lambda \in (-1,1]$ so that $\dot{U}(1)= \frac{1-\lambda}{1+\lambda}$, where $\dot{U} (1)$ is the initial speed of harmonic evolution, defined by~\eqref{uprime}. Theorem~\ref{thnew} yields
\[
R_\ast (h) \ge  \frac{R^2+\lambda}{(1+\lambda)R} \ge  \frac{1}{2} \left(R+ \frac{1}{R} \right).
\]
The latter inequality turns into identity when  $\lambda=1$; that is, for free harmonic evolutions.
\end{proof}

\begin{remark}\label{rmuni1}
Suppose $h\in \Ha_D(\A, \A^\ast)$ with $R_\ast= \frac{1}{2} \left(R+\frac{1}{R}\right)$.  Then
\[U(R) \le R_*^2= \left(\frac{R^2+1}{2R}\right)^2.\]
In view of~\eqref{blah} this is only possible if $\lambda=1$ and equality holds in~\eqref{u1} for all $\rho \in (1,R)$.  Therefore, \eqref{u0} also turns into identity for all $\rho \in (1,R)$.  As  in the proof of Theorem~\ref{thnew}, the equality case of Proposition~\ref{varprop1} implies
\begin{equation}\label{equni1}
\Ha_D(\A, \A^\ast) = \{\alpha h^1 \colon \abs{\alpha}=1\}.
\end{equation}
\end{remark}

The case $\lambda=0$ of Theorem~\ref{thnew} also gains in interest if we combine it   with Lemma~\ref{uclem2}. We obtain a refinement of Schottky's theorem.

\begin{corollary}\label{Schottky}
Let $h\colon A(1,R)\to\As$ be a conformal mapping such that
\begin{equation}\label{aftervar7}
\lim_{\rho\searrow 1}\dashint_{C_\rho}h=0.
\end{equation}
Then
\begin{equation}\label{aftervar8}
R_*(h):=\lim_{\rho\nearrow R}\left(\dashint_{C_{\rho}}\abs{h}^2\right)^{1/2}\ge R.
\end{equation}
Also,
the area of the target annulus $\As$ is not smaller than that of the domain $\A=A(1,R)$.
\end{corollary}

\begin{proof} Inequality~\eqref{aftervar8} is obtained from Proposition~\ref{mean0} be setting $\lambda=0$.
To estimate the area of $\As$, we consider the Laurent expansion of $h$ around zero,
\begin{equation*}
h(z)=\sum_{n=-\infty}^{\infty}a_n z^n,\quad\text{where }\ a_0=\lim_{\rho\searrow 1}\dashint_{C_\rho}h=0.
\end{equation*}
We explore the orthogonality of the powers of $z$ to compute
\begin{equation*}
R_*^2(h)-1=\lim_{\rho\nearrow R}\dashint_{C_{\rho}}\abs{h}^2-\lim_{\rho\searrow 1}\dashint_{C_{\rho}}\abs{h}^2
=\sum_{n\ne 0}\abs{a_n}^2(R^{2n}-1),
\end{equation*}
where the sum is taken to be $\infty$ when the series diverges. Hence, by~\eqref{aftervar8},
\begin{equation}\label{add1}
\sum_{n\ne 0}\abs{a_n}^2(R^{2n}-1)\ge R^2-1.
\end{equation}
On the other hand, the area of $\As$ is equal to
\begin{equation*}\begin{split}
\iint_{\A}\abs{h'(z)}^2&=\pi\sum_{n\ne 0}n\abs{a_n}^2(R^{2n}-1)  \\
&\ge \pi \sum_{n\ne 0}\abs{a_n}^2(R^{2n}-1),
\end{split}\end{equation*}
which is greater than or equal to the area of $\A$, by virtue of~\eqref{add1}.
Here we used the inequality $n(R^{2n}-1)\ge (R^{2n}-1)$, which is valid for all integers $n\ne 0$.
\end{proof}

\begin{remark}
The observant reader may notice notice that $\LL^\lambda [U] \ge 0$ for the quadratic mean $U=U(\rho, h)$ of a polar mapping; that is, $\abs{h(z)}= \const$ on every circle $C_\rho$. Indeed, \eqref{uc4} and H\"older's inequality imply
\[\LL^\lambda [U] \ge  \frac{2}{\rho^2}\,\dashint_{C_{\rho}}\left(\abs{h_{\theta}}^2-\abs{h}^2 \right) \ge  \frac{2}{\rho^2} \left[  \left( \dashint_{C_{\rho}} \abs{h_{\theta}}\right)^2- \dashint_{C_{\rho}}  \abs{h}^2
\right]=0. \]
Here $\frac{1}{\rho} \int_{C_\rho} \abs{h_\theta}$ represents the length of the curve $C^\ast_\rho = h(C_\rho)$ which equals $2 \pi h(\rho)$. The Nitsche bound for polar mappings then follows.

\end{remark}

\section{Proof of Theorem~\ref{thgene1}}
\subsection{The case $1<R\le e$}\label{secnbproof}

We will prove the following, more precise statement: if for some $-1<\lambda\le 1$
\begin{equation}\label{ass1}
\dot{U}(1)\ge 2\frac{1-\lambda}{1+\lambda}
\end{equation}
and
\begin{equation}\label{ass2}
R^2-1-(R^2-\lambda)\log R\ge 0,
\end{equation}
then
\begin{equation}\label{Re5}
R_*(h)\ge \frac{R^2+\lambda}{(1+\lambda)R}.
\end{equation}
Theorem~\ref{thgene1} for $1<R\le e$ follows by choosing $\lambda=1$ above.
We may assume that $R_*(h)<\infty$, otherwise~\eqref{Re5} is vacuous.

In this proof we examine $\LL^{\lambda}[U]$ for the quadratic mean $U(\rho)=\dashint_{C_{\rho}}\abs{h}^2$. Examples show that for a general harmonic mapping $h\in \Ha (\A, \As)$, the operator
$\LL^{\lambda}[U]$ need not be nonnegative pointwise in the entire interval $(1,R)$. But it does in an average sense,
when the domain annulus $\A=A(1,R)$ is not too wide. We shall integrate $\LL^{\lambda}[U]$ against a carefully adopted weight
in the interval $(1,R)$. It is important that such a weighted average of $\LL^{\lambda}[U]$ will depend only on $U(1)$, $\dot{U}(1)$ and $U(R)$, which we defined in Section~\ref{seccm}. The following integral is the key ingredient,
\begin{equation}\label{Re1}
\K^{\lambda}[U]:=\int_1^R\frac{\rho(R^2-\rho^2)}{\rho^2+\lambda}\LL^{\lambda}[U]\,d\rho.
\end{equation}
This, possibly improper, integral has a well-defined value in $(-\infty,+\infty]$ because
\[
\LL^{\lambda}[U]=\LL^{\lambda}[V]+\LL^{\lambda}[U_0]
\]
where $\LL^{\lambda}[V]\ge 0$ by Proposition~\ref{varprop1} and $\LL^{\lambda}[U_0]$ is bounded
on $[1,R]$. We integrate~\eqref{Re1} by parts
using the divergence form~\eqref{uc2} of $\LL^{\lambda}[U]$ to obtain the identity
\begin{equation}\begin{split}\label{Re2}
\K^{\lambda}[U]&=\int_1^R\frac{R^2-\rho^2}{\rho^2}\frac{d}{d\rho}\left[\rho^3\frac{d}{d\rho}\left(\frac{U}{\rho^2+\lambda}\right)\right]\,d\rho
\\
&=\frac{2R^2}{R^2+\lambda}U(R)-2\frac{\lambda R^2+1}{(1+\lambda)^2}U(1)-\frac{R^2-1}{1+\lambda}\dot U(1).
\end{split}\end{equation}
Since $U(1)=1$,~\eqref{Re2} takes the form
\begin{equation}\begin{split}\label{Re4}
\K^{\lambda}[U]&=\frac{2R^2}{R^2+\lambda}\left\{U(R)-\left[\frac{R^2+\lambda}{(1+\lambda)R}\right]^2\right\} \\
&-\frac{R^2-1}{1+\lambda}\left(\dot U(1)-2\frac{1-\lambda}{1+\lambda}\right).
\end{split}\end{equation}
The desired bound~\eqref{Re5} will follow once we know that
\begin{equation}\label{Re5prime}
\K^{\lambda}[U]\ge 0.
\end{equation}
Before proving~\eqref{Re5prime}, observe that equality occurs for the mapping
\begin{equation}\label{Re7}
h^{\lambda}(z)=\frac{1}{1+\lambda}\left(z+\frac{\lambda}{\bar z}\right)
=\left(\frac{\rho}{1+\lambda}+\frac{\lambda}{(1+\lambda)\rho}\right)e^{i\theta}
\end{equation}
because $\LL^{\lambda}[U(\rho,h^{\lambda})]\equiv 0$, see Section~\ref{secuc}.
This suggests writing our mapping $h$ in the form
\begin{equation}\label{Re8}
h(\rho e^{i\theta})=h^{\lambda}(\rho e^{i\theta}) g(\rho e^{i\theta}),\quad 1<\rho<R.
\end{equation}
We need a lemma.

\begin{lemma}\label{windh}
Let $h\in\Ha(\A,\As)$, then
\begin{equation}\label{Re9a}
\lim_{\rho\searrow 1}\dashint_{C_{\rho}}\im \bar hh_{\theta}=1,
\end{equation}
hence
\begin{equation}\label{Re9b}
\lim_{\rho\searrow 1}\dashint_{C_{\rho}}\im \bar gg_{\theta}=0.
\end{equation}
\end{lemma}

\begin{proof}
To prove~\eqref{Re9a}
we compute the area of a domain $\Omega_\rho$ inside the Jordan curve $C_\rho^\ast= h(C_\rho)$ as follows.
\[\abs{\Omega_\rho} = \frac{1}{2i} \int_{C_\rho} \bar h \, \dtext h = \frac{1}{2i} \int_{0}^{2\pi}  \bar h h_\theta\, \dtext \theta = \pi \dashint_{C_\rho} \im \bar h h_\theta . \]
On other hand  $\abs{\Omega_\rho}-\pi$ is the area of the region enclosed between $C_1$ and $C_\rho^\ast$, which converges to $0$ as $\rho \searrow 1$.

To prove~\eqref{Re9b} we differentiate   the mapping $g=h/h^{\lambda}$ and find that
\begin{equation*}
\bar g g_{\theta}=\abs{h^{\lambda}}^{-2}\left[\bar h h_{\theta}-\abs{h}^2\frac{h_{\theta}^{\lambda}}{h^{\lambda}}\right].
\end{equation*}
This implies
\begin{equation*}
\begin{split}
\lim_{\rho\searrow 1} \dashint_{C_{\rho}} \im \bar g g_{\theta} &=
\lim_{\rho\searrow 1} \frac{(1+\lambda)^2\rho^2}{(\rho^2+\lambda)^2}\left[\dashint_{C_{\rho}}\im \bar h h_{\theta}
-\dashint_{C_{\rho}}\abs{h}^2\im\frac{h_{\theta}^{\lambda}}{h^{\lambda}}\right] \\
&=\lim_{\rho\searrow 1}\dashint_{C_{\rho}}\im \bar h h_{\theta}-\dashint_{C_1}\im\frac{h_{\theta}^{\lambda}}{h^{\lambda}}
=1-1=0. \qedhere
\end{split} \end{equation*}
\end{proof}

The same circular means, $\dashint_{C_{\rho}}\bar g g_{\theta}$, link us with the Jacobian determinant
\begin{equation*}
J(z,g)=\frac{1}{\rho}\im \bar g_{\rho} g_{\theta}=\abs{g_z}^2-\abs{g_{\bar z}}^2.
\end{equation*}
Indeed, we differentiate with respect to $\rho$ and integrate by parts along the circle $C_{\rho}$ to obtain
\begin{equation*}\begin{split}
\frac{d}{d\rho}\dashint_{C_{\rho}}\bar g g_{\theta}&= \dashint_{C_{\rho}} (\bar g_{\rho}g_{\theta}+ \bar g g_{\theta\theta})
=\dashint_{C_{\rho}}(\bar g_{\rho}g_{\theta}-\bar g_{\theta}g_{\rho}) \\
&=2i\dashint_{C_{\rho}}\im \bar g_{\rho} g_{\theta}.
\end{split}\end{equation*}
Hence the formula
\begin{equation}\label{Re10}
\frac{d}{d\rho}\dashint_{C_{\rho}}\im \bar g g_{\theta}=\frac{1}{\pi}\int_{C_{\rho}}(\abs{g_z}^2-\abs{g_{\bar z}}^2).
\end{equation}
We now take advantage of formula~\eqref{uc4} for the operator $\LL^{\lambda}$,
\begin{equation}\label{Re11}
\LL^{\lambda}[U]=\frac{2}{\rho^2}\dashint_{C_{\rho}}\left\{\abs{h_{\theta}}^2-\abs{h}
+(\rho^2+\lambda)^2\left|\frac{d}{d\rho}\left(\frac{\rho h}{\rho^2+\lambda}\right)\right|^2\right\}.
\end{equation}
In order to express $\LL^{\lambda}[U]$ by means of $g$ we compute the terms under the integral sign,
\begin{equation*}\begin{split}
\abs{h}^2&=\frac{(\rho^2+\lambda)^2}{(1+\lambda)^2\rho^2}\abs{g}^2; \\
\abs{h_{\theta}}^2&=\frac{(\rho^2+\lambda)^2}{(1+\lambda)^2\rho^2}\abs{g_{\theta}+ig}^2; \\
\frac{d}{d\rho}\left(\frac{\rho h}{\rho^2+\lambda}\right)&=\frac{e^{i\theta}}{1+\lambda}g_{\rho}.
\end{split}\end{equation*}
Therefore,
\begin{equation*}\begin{split}
\LL^{\lambda}[U]&=\frac{2}{\rho^4}\frac{(\rho^2+\lambda)^2}{(1+\lambda)^2}
\dashint_{C_{\rho}}\left(\abs{g_{\theta}+ig}^2-\abs{g}^2+\rho^2\abs{g_{\rho}}^2\right) \\
&=\frac{2(\rho^2+\lambda)^2}{(1+\lambda)^2\rho^2}
\dashint_{C_{\rho}}\left(\abs{g_{\rho}}^2+\rho^{-2}\abs{g_{\theta}}^2+2\rho^{-2}\im (\bar g g_{\theta})\right) \\
&=\frac{4(\rho^2+\lambda)^2}{(1+\lambda)^2\rho^2}
\dashint_{C_{\rho}}\left(\abs{g_z}^2+\abs{g_{\bar z}}^2+\rho^{-2}\im (\bar g g_{\theta})\right).
\end{split}\end{equation*}
Substitute this into~\eqref{aftervar1} to obtain
\begin{equation}\label{Re12}
\K^{\lambda}[U]=I+II,
\end{equation}
where
\begin{equation*}\begin{split}
I&=\frac{4}{(1+\lambda)^2}\int_1^R\frac{(R^2-\rho^2)(\rho^2+\lambda)}{\rho}\dashint_{C_{\rho}}\left(\abs{g_z}^2+\abs{g_{\bar z}}^2\right);
\\
II&=\frac{4}{(1+\lambda)^2}\int_1^R\frac{(R^2-\rho^2)(\rho^2+\lambda)}{\rho^3}\dashint_{C_{\rho}}\im(\bar g g_{\theta}).
\end{split}\end{equation*}
By Fubini's theorem $I$ takes the form of a double integral
\begin{equation*}
I=\frac{4}{(1+\lambda)^2\pi}\iint_{\A}\frac{(R^2-\rho^2)(\rho^2+\lambda)}{2\rho^2}\left(\abs{g_z}^2+\abs{g_{\bar z}}^2\right).
\end{equation*}
Before converting $II$ into double integral we shall first integrate by parts. For this we express the factor in front of the circular mean
as
\begin{equation}\label{Re13}
\frac{(R^2-\rho^2)(\rho^2+\lambda)}{\rho^3}=-\frac{d}{d\rho}\left[(R^2-\lambda)\log\frac{R}{\rho}-\frac{(R^2-\rho^2)(\rho^2-\lambda)}{2\rho^2}\right].
\end{equation}
The expression in the square brackets vanishes at the endpoint $\rho=R$, whereas
$\lim_{\rho\searrow 1} \dashint_{C_{\rho}} \im \bar g g_{\theta}=0$
by~\eqref{Re9b}. Therefore, integration by parts will not produce the endpoint terms. In view of formula~\eqref{Re10} we obtain
\begin{equation*}
II=\frac{4}{(1+\lambda)^2\pi}\int_1^R \left[(R^2-\lambda)\log\frac{R}{\rho}-\frac{(R^2-\rho^2)(\rho^2-\lambda)}{2\rho^2}\right]
\int_{C_{\rho}}\left(\abs{g_z}^2-\abs{g_{\bar z}}^2\right).
\end{equation*}
Adding up $I$ and $II$ we arrive at the formula
\begin{equation}\begin{split}\label{Re14}
\K^{\lambda}[U]&=\frac{4}{(1+\lambda)^2\pi}\iint_{\A}\left[(R^2-\lambda)\log\frac{R}{\rho}+\frac{(R^2-\rho^2)\lambda}{\rho^2}\right]\abs{g_z}^2\\
&+\frac{4}{(1+\lambda)^2\pi}\iint_{\A}\left[(R^2-\rho^2)-(R^2-\lambda)\log\frac{R}{\rho}\right]\abs{g_{\bar z}}^2.
\end{split}\end{equation}
We leave to the reader a routine task of verifying that the factor in from of $\abs{g_z}^2$ is nonnegative; that is,
\begin{equation}\label{Re15}
(R^2-\lambda)\log\frac{R}{\rho}+\frac{(R^2-\rho^2)\lambda}{\rho^2}\ge 0,
\end{equation}
whenever $1\le\rho\le R$ and $-1<\lambda\le 1$.

To establish the inequality~\eqref{Re5prime} it suffices to ensure that
\begin{equation*}
(R^2-\rho^2)-(R^2-\lambda)\log\frac{R}{\rho}\ge 0.
\end{equation*}
This expression, regarded as a function in $1\le\rho\le R$, is concave, vanishes at $\rho=R$,
and is nonnegative at $\rho=1$ by virtue of~\eqref{ass2}.
This completes the proof of Theorem~\ref{thgene1} in case $1<R\le e$.
\qed
\begin{remark}
As $\lambda$ decreases from $1$ to $-1$ the condition~\eqref{ass2} becomes more restrictive but it still holds for $R$ sufficiently close to $1$. For example, if $\lambda =0$, then~\eqref{ass2} certainly holds whenever $1<R\le 2$.
\end{remark}

\begin{remark}\label{rmuni2}
Suppose $1<R \le e$ and $h\in \Ha (\A, \A^\ast)$ with $R_\ast= \frac{1}{2} \left(R+\frac{1}{R}\right)$.  Since
$U(R) \le R_*^2$,
we must have $\lambda=1$ in~\eqref{Re5}. Also, $\K^1[U]\le 0$ because of~\eqref{Re4}.  By~\eqref{Re14} we have $g=\const$. Thus
\begin{equation}\label{equni2}
\Ha (\A, \A^\ast) = \{\alpha h^1 \colon \abs{\alpha}=1\} \quad \mbox{ where } 1<R \le e.
\end{equation}
\end{remark}

\subsection{The case $ e<R\le e^{3/2}$}\label{secnbproof2}

In this case we rely heavily on the orthogonal decomposition~\eqref{hfeq1}. The operator $\LL^{\lambda}$ and associated integral $\K^{\lambda}$ from the previous subsection will be used here only with $\lambda=1$ and denoted
simply as $\LL$ and $\K$.
Let us state here the relevant versions of identities~\eqref{Re1} and~\eqref{Re2}, namely
\begin{equation}\label{Re1r}
\K[U]:=\int_1^R\frac{\rho(R^2-\rho^2)}{\rho^2+1}\LL[U]\,d\rho
\end{equation}
and
\begin{equation}\label{Re2r}
\K[U]=\frac{2R^2}{R^2+1}U(R)-\frac{R^2+1}{2}U(1)-\frac{R^2-1}{2}\dot U(1).
\end{equation}

We require the following lemma, whose proof is postponed to the end of the section.

\begin{lemma}\label{LVest}
Suppose that $R>e$ and $h\in \Ha(\A, \As)$. Then
\begin{equation}\label{intLV}
\K[V]\ge (R^2-1)\abs{b_0}^2.
\end{equation}
\end{lemma}

\begin{proof}[Proof of Theorem~\ref{thgene1} for $e<R\le e^{3/2}$]
Inequality~\eqref{intLV} yields
\begin{equation}\label{p1}
\K[U]=\K[V]+\K[U_0]\ge (R^2-1)\abs{b_0}^2+\K[U_0].
\end{equation}
From~\eqref{Re2r} we have
\begin{equation*}
\K[U_0]=
\frac{2R^2}{R^2+1}\abs{a_0\log R+b_0}^2-\frac{R^2+1}{2}\abs{b_0}^2 - 2\re(a_0\bar b_0)\frac{R^2-1}{2},
\end{equation*}
hence
\begin{equation}\begin{split}\label{p3}
\K[U]&\ge\frac{2R^2\log^2R}{R^2+1}\abs{a_0}^2 +\frac{R^4+2R^2-3}{2(R^2+1)}\abs{b_0}^2 \\
&+2\re(a_0\bar b_0)  \left(\frac{2R^2\log R}{R^2+1}-\frac{R^2-1}{2}\right).
\end{split}\end{equation}
Let us record for future use that~\eqref{p3} is valid whenever $R>e$, as the condition  $R\le e^{3/2}$ was not used yet.

The quadratic form with respect to $a_0$ and $b_0$ in the righthand side of~\eqref{p3} is positive  definite, provided that the quantity
\begin{equation}\label{p4}
\left(\frac{2R^2\log^2R}{R^2+1}\right)\left(\frac{R^4+2R^2-3}{2(R^2+1)}\right)-\left(\frac{2R^2\log R}{R^2+1}-\frac{R^2-1}{2}\right)^2
\end{equation}
is positive. Multiplying~\eqref{p4} by $4(R^2+1)$, we arrive at the function
\begin{equation}\label{p5}
\phi(R):=4R^2(R^2-3)\log^2R+8R^2(R^2-1)\log R-(R^2-1)(R^4-1)
\end{equation}
It remains to prove that  $\phi(R)>0$ for $e\le R\le e^{3/2}$. First compute
\[\phi(e)=13e^4-e^6-19e^2-1>0\quad \text{and }\quad \phi(e^{3/2})=22e^6-e^9-38e^3-1>0.\]
Since the second derivative
\[\begin{split}
\frac{d^2}{dR^2}(R^{-4}\phi(R))=&-\frac{2}{R^6}\big\{4R^4\log R+36R^2(\log^2R-\log R)\\ &+R^2(R^4-11)+10\big\}
\end{split}\]
is negative for $R\ge e$, it follows that $\phi(R)> 0$ for $e\le R\le e^{3/2}$.

Thus, $\K[U]\ge 0$, which by~\eqref{Re2r} yields
\[\frac{2R^2}{R^2+1}U(R)\ge \frac{R^2+1}{2}U(1)+\frac{R^2-1}{2}\dot U(1)\ge \frac{R^2+1}{2},\]
as required.\end{proof}

\begin{remark}\label{rmuni3}
Suppose $e< R \le e^{3/2}$ and $h\in \Ha (\A, \A^\ast)$ with $R_\ast= \frac{1}{2} \left(R+\frac{1}{R}\right)$.  Since $U(R) \le R_*^2$,
we have $\K[U]\le 0$ because of~\eqref{Re2r}. On the other hand, the quadratic form in~\eqref{p3} is strictly positive unless
$a_0=b_0=0$. Invoking Remark~\ref{rmuni1}, we arrive at
\begin{equation}\label{equni3}
\Ha (\A, \A^\ast) = \Ha_D (\A, \A^\ast)= \{\alpha h^1 \colon \abs{\alpha}=1\} \quad \mbox{ where } e<R \le e^{3/2}.
\end{equation}
\end{remark}

\begin{proof}[Proof of Lemma~\ref{LVest}]
Let us assume for now that $h$ is continuously differentiable up to the inner circle $C_1$; this
assumption will be removed later. It is easy to see that
\begin{equation}\label{smoothh}
\frac{1}{i}\dashint_{C_1}\bar h h_{\theta}-\dashint_{C_1}\abs{h}^2+\left|\dashint_{C_1}h\right|^2
=\sum_{n\ne 0}(n-1)\abs{a_n+b_n}^2.
\end{equation}
We claim that
\begin{equation}\label{intLUn}
\int_1^R\frac{\rho(R^2-\rho^2)}{\rho^2+1}\LL[U_n]\,d\rho\ge (R^2-1)(n-1)\abs{a_n+b_n}^2, \quad n\ne 0.
\end{equation}
Indeed, in Section~\ref{secvs} we found that the lefthand side of~\eqref{intLUn}
is nonnegative for $n\ne 0$. Thus, we only need to prove~\eqref{intLUn} for $n\ge 2$. Using
the identity~\eqref{Re2r}, we find
\begin{equation}\label{p10}\begin{split}
\int_1^R &\frac{\rho(R^2-\rho^2)}{\rho^2+1}\LL[U_n]\,d\rho-(R^2-1)(n-1)\abs{a_n+b_n}^2 \\
&=\frac{1}{2(R^2+1)}\left\{A_n\abs{a_n^2}+B_n\abs{b_n}^2+2C_n\re(a_n\bar b_n)\right\},
\end{split}\end{equation}
where
\begin{equation}\begin{split}\label{ABC}
A_n&=4R^{2n+2}+(R^2-3)(R^2+1)-4n(R^4-1); \\
B_n&=4R^{2-2n}+(R^2-3)(R^2+1); \\
C_n&=-(R^2-1)(2n(R^2+1)-R^2-3).
\end{split}\end{equation}
Our goal is to show that the quadratic form in~\eqref{p10} is positive definite as long as
$n\ge 2$ and $R\ge e$. To this end, we can replace the coefficient $B_n$ with the smaller quantity
$\widetilde{B}_n=(R^2-3)(R^2+1)$. Since $R^2\ge e^2>3$, we have $\widetilde{B}_n>0$. Therefore, it remains
to prove that
\begin{equation}\label{dn1}\begin{split}
D(n,R)&:=A_n\widetilde{B}_n-C_n^2>0\quad \text{for }R\ge e,\ n\ge 2.
\end{split}\end{equation}
After a simplification,
\begin{equation*}
\begin{split}
D(n,R)= & \,4\big\{R^{2n+2}(R^4-2R^2-3)-n^2R^8 +(4n-2)R^6\\  &   +2n^2R^4+(6-4n)R^2-n^2\big\}.
\end{split}
\end{equation*}
First consider the case $n=2$:
\begin{equation*}
D(2,R)=4(R^2-1)(R^8-5R^6-2R^4+6R^2+4)>0
\end{equation*}
because $R^8\ge e^2R^6>7R^6$.
We will show that $D(n,R)$ is convex and increasing with respect to $n\ge 2$ for each $R\ge e$.
Indeed
\begin{equation*}
\frac{\D D(n,R)}{\D n}=8\big\{R^{2n+2}(R^4-2R^2-3)\log R-nR^8+2R^6+2nR^4-2R^2-n\big\}.
\end{equation*}
This derivative is positive at $n=2$, when it simplifies to
\begin{equation*}
\begin{split}
8&\left\{R^6(R^4-2R^2-3)\log R-2R^8+2R^6+4R^4-2R^2-2\right\} \\
&\ge 8\left\{R^6(R^4-2R^2-3)-2R^8+2R^6+4R^4-2R^2-2\right\} \\
&=8(R^2+1)(R^8-5R^6+4R^4-2)>0.
\end{split}\end{equation*}
This leads us to consider the second derivative
\begin{equation*}
\frac{\D^2 D(n,R)}{\D n^2}=16R^{2n+2}(R^4-2R^2-3)\log^2R-8(R^4-1)^2.
\end{equation*}
Since $R^4-2R^2-3=(R^2-3)(R^2+1)>0$, the second derivative is increasing with $n$. For $n=2$ it is equal to
\begin{equation*}
\begin{split}
&16R^{6}(R^4-2R^2-3)\log^2R-8(R^4-1)^2 \\
&\ge 16(R^{10}-2R^{8}-3R^6)-8(R^4-1)^2 \\
&=8(R^2+1)(2R^8-7R^6+R^4+R^2-1)
\end{split}\end{equation*}
which is positive since $R^8>7R^6$. Thus, $D(n,R)$ is convex and increasing with respect to $n\ge 2$.
This completes the proof of~\eqref{dn1} and therefore of~\eqref{intLUn}.

Summing~\eqref{intLUn} over $n\ne 0$ and using~\eqref{smoothh}, we obtain
\begin{equation}\label{smooth2}
\int_1^R\frac{\rho(R^2-\rho^2)}{\rho^2+1}\LL[V]\,d\rho\ge
(R^2-1)\left\{\frac{1}{i}\dashint_{C_1}\bar h h_{\theta}-\dashint_{C_1}\abs{h}^2+\left|\dashint_{C_1}h\right|^2\right\}
\end{equation}
By Lemma~\ref{windh} the righthand side of~\eqref{smooth2} is equal to $(R^2-1)\abs{b_0}^2$.

Now we remove the assumption that $h$ is smooth up to $C_1$. For $r\in (1,R/e)$ we can apply~\eqref{smooth2}
to the mapping $f\colon A(1,R/r)\to\C$ defined by $f(z)=h(rz)$. Using Lemma~\ref{windh}, we
conclude that
\[
\left( \frac{1}{i}\dashint_{C_1}\bar f f_{\theta}-\dashint_{C_1}\abs{f}^2+\left|\dashint_{C_1}f\right|^2 \right)
\longrightarrow \left|\dashint_{C_1}h\right|^2=\abs{b_0}^2
\]
as $r\searrow 1$. Also, substitution $\rho=t/r$ yields
\begin{equation*} \int_1^{R/r}\frac{\rho ((R/r)^2-\rho^2)}{\rho^2+1}\LL[V(\rho,f)]\,d\rho
=\frac{1}{r^2}\int_r^{R}\frac{t(R^2-t^2)}{t^2+r^2}\LL[V(t,h)]\,dt.
\end{equation*}
Recall that $\LL[V]\ge 0$ by Proposition~\ref{varprop1}.
Using the monotone convergence theorem, we conclude that
\[
\int_r^{R}\frac{t(R^2-t^2)}{t^2+r^2}\LL[V(t,h)]\,dt  \to \int_1^{R}\frac{t(R^2-t^2)}{t^2+1}\LL[V(t,h)]\,dt.
\]
as $r\searrow 1$. Thus, inequality~\eqref{smooth2} remains true without the smoothness assumption on $h$.
\end{proof}

\section{Proof of Theorem~\ref{thgene2}~\eqref{thgene22}}
The case $R\le e^{3/2}$ was already covered by Theorem~\ref{thgene1}. Thus we assume that $R>e$, in which case~\eqref{p3} is known to be true.
Since $h\in \Ha_N(\A,\As)$, we have $a_0=\dashint_{C_\rho}h_{\rho}=0$. Inequality~\eqref{p3} takes the form
\begin{equation}\label{notme}
\K[U]\ge \frac{R^4+2R^2-3}{2(R^2+1)}\abs{b_0}^2 \ge 0.
\end{equation}
It follows from~\eqref{Re2r} that
\[\frac{2R^2}{R^2+1}U(R) = \frac{R^2+1}{2}U(1)+ \frac{R^2-1}{2}\dot{U}(1)+\K^1[U]\ge \frac{R^2+1}{2}, \]
hence
\[U(R)\ge\left(\frac{R^2+1}{2R}\right)^2 \]
as required.\qed

\begin{remark}\label{rmuni4}
Suppose   $h\in \Ha_N (\A, \A^\ast)$ with $R_\ast= \frac{1}{2} \left(R+\frac{1}{R}\right)$.  Since
$U(R) \le R_*^2$, we have $\K[U]\le 0$ because of~\eqref{Re2r}. Contrasting this with~\eqref{notme}, we are led to the conclusion $b_0=0$. Invoking Remark~\ref{rmuni1}, we arrive at
\begin{equation}\label{equni4}
\Ha_N (\A, \A^\ast) = \Ha_D (\A, \A^\ast)= \{\alpha h^1 \colon \abs{\alpha}=1\}.
\end{equation}
\end{remark}
\section{Proof of Theorem~\ref{thmuni}}
Combining~\eqref{equni1},~\eqref{equni2},~\eqref{equni3} and~\eqref{equni4}, Theorem~\ref{thmuni} follows. \qed

\bibliographystyle{amsplain}

\begin{thebibliography}{9}

\bibitem{AIM}
K. Astala, T. Iwaniec, and G. J. Martin,  \textit{Deformations of
smallest mean distortion}. Arch. Ration. Mech. Anal., to appear.


\bibitem{Dub}  P. Duren,  \textit{Harmonic mappings in the plane}, Cambridge Tracts in Mathematics, 156. Cambridge University
Press, Cambridge, 2004.

\bibitem{IMS}
T. Iwaniec, G. Martin, and C. Sbordone, \textit{$L^p$-integrability \& weak type $L^2$-estimates for the gradient of harmonic mappings of $\mathbb{D}$.} Discrete Contin. Dyn. Syst. Ser. B 11 (2009), no. 1, 145--152.

\bibitem{IO}
T. Iwaniec  and  J. Onninen,  \textit{$n$-Harmonic mappings between
annuli.}  Preprint.

\bibitem{Ka}
D. Kalaj,  \textit{On the Nitsche conjecture for harmonic mappings in $\R^2$ and $\R^3$.} Israel J. Math. {\bf 150} (2005), 241--251.

\bibitem{Ly}
A. Lyzzaik,  \textit{The modulus of the image annuli under univalent harmonic mappings and a conjecture of J.C.C. Nitsche},  J. London Math. Soc., {\bf 64} (2001), 369--384.

\bibitem{Nebook}
M. H. A. Newman, \textit{Elements of the topology of plane sets of points},
2nd ed. Cambridge University Press, 1951.

\bibitem{Ni}
J. C. C. Nitsche,   \textit{On the modulus of doubly connected regions under harmonic mappings},  Amer. Math. Monthly,  {\bf 69} (1962), 781--782.

\bibitem{Sc}
F. H. Schottky,  \textit{\"Uber konforme Abbildung von mehrfach
zusammenh\"angenden Fl\"ache.}  J. f\"ur Math., {\bf 83} (1877).

\bibitem{We}  A. Weitsman,  \textit{Univalent harmonic mappings of annuli and a conjecture of J.C.C. Nitsche}, Israel J. Math., {\bf 124}  (2001),  327--331.


\end{thebibliography}

\end{document}